\documentclass[10pt]{article}
\usepackage{mathrsfs}
\usepackage{latexsym}
\usepackage{bbm}
\usepackage{amsthm}
\usepackage{amsmath}
\usepackage{amsfonts}
\usepackage{amssymb}
\usepackage{multirow}
\usepackage{latexsym}
\usepackage[dvips]{graphicx}
\usepackage{overpic}
\usepackage{graphicx}

\newtheorem{theorem}{\bf Theorem}[section]
\newtheorem*{thA}{\bf Theorem A}{\vspace{2ex}}
\newtheorem{lemma}[theorem]{\bf Lemma}
\newtheorem{prop}[theorem]{\bf Proposition}

\newenvironment{definition}{{\bf Definition}}{\vspace{2ex}}
\newenvironment{notation}{{\bf Notation}}{\vspace{2ex}}

\def \and {\, \mb
ox{\rm and}\, }

\def \supp {\,{\rm supp}\,}

\def \l {\left}
\def \r {\right}
\makeatletter
\def\R{\mathbb{R}^n}
\newcommand\diff{\,\mathrm{d}}

\newcommand{\Rmnum}[1]{\expandafter\@slowromancap\romannumeral #1@}
\newcommand{\norm}[1]{\Vert#1\Vert}
\newcommand{\abs}[1]{\left\vert#1\right\vert}

\makeatother
\begin{document}

\title {\bf The sharp $L^p$ decay of oscillatory integral operators with certain homogeneous polynomial phases in several variables}

\author{Shaozhen Xu\thanks{Corresponding author}\, \thanks{School of Mathematical Sciences,
        University of Chinese
        Academy of Sciences, Beijing 100049, P.R. China. E-mail address:
        {\it xushaozhen14b@mails.ucas.ac.cn}.}\,\,\,,\quad
       Dunyan Yan\thanks{School of Mathematical Sciences, University of
        Chinese Academy of Sciences, Beijing 100049, P. R. China. E-mail
        address: {\it ydunyan@ucas.ac.cn}.} }

\date{}
\maketitle{}

\begin{abstract}
 We obtain the $L^p$ decay of oscillatory integral operators $T_\lambda$ with certain homogeneous polynomial phase of degree $d$ in $(n+n)$-dimensions. In this paper we require that $d>2n$. If $d/(d-n)<p<d/n$, the decay is sharp and the decay rate is related to the Newton distance. In the case of $p=d/n$ or $d/(d-n)$, we also obtain the \textit{almost} sharp decay, here ``\textit{almost}" means the decay contains a $\log(\lambda)$ term. For otherwise, the $L^p$ decay of $T_\lambda$ is also obtained but not sharp. A counterexample also arises in this paper to show that $d/(d-n)\leq p\leq d/n$ is not necessary to guarantee the sharp decay.\\
\end{abstract}

\section{Introduction}
We consider the following oscillatory operator:
\begin{equation}\label{oio}
T_\lambda(f)(x)=\int_{\R} e^{i\lambda S(x,y)}\psi(x,y)f(y)\diff{y},\ n\geq 2
\end{equation}
where $x\in \R, \psi(x,y)$ is a smooth function supported in a compact neighborhood of the origin, $S(x,y)=\sum_{|\alpha|+|\beta|=d}a_{\alpha,\beta}x^\alpha y^\beta$ is a real-valued homogeneous polynomial in higher dimension with degree $d$, here $\alpha, \beta$ are multi-indices. Research on this operator centers on the decay of $L^p$ bound as the parameter $\lambda$ tends to infinity. In one dimensional case, Phong and Stein contributed a lot to this subject. In a series of their articles \cite{phong1992oscillatory}, \cite{phong1994models}, \cite{phong1997newton}, \cite{phong1998damped}, they developed the almost-orthogonality method to obtain the sharp $L^2$ decay of oscillatory integral operators with phase functions varying from homogeneous polynomials to real-valued analytic functions. They also clarified the relation between the decay rate and the Newton distance raised by Arnold and Varchenko in \cite{arnold1986singularites}. Later, the sharp $L^2$ estimate was extended to $C^{\infty}$ phases by Rychkov \cite{rychkov2001sharp} and Greenblatt \cite{greenblatt2005sharp}. When $S$ is smooth and $T_\lambda$ has two-sided Whitney fold, Greenleaf and Seeger obtained the endpoint estimates for the $L^p$ decay rate of $T_\lambda$ in \cite{greenleaf1999oscillatory}. Yang obtained the sharp endpoint estimate in \cite{yang2004sharp} with the assumption $a_{1,d-1}a_{d-1,1}\neq 0$, here $a_{\alpha,\beta}$ are coefficients of homogeneous polynomial phase function $S(x,y)$ in $\mathbb{R}\times \mathbb{R}$. Shi and Yan \cite{shi2016sharp} established the sharp endpoint $L^p$ decay for arbitrary homogeneous polynomial phase functions. Later, Xiao extended this result to arbitrary analytic phases as well as presented a very specific review for this subject in \cite{xiao2016endpoint}. Higher dimensional case even $L^2$ estimate has not been understood well. The one dimensional result of $L^2$ decay has been partially extended to (2+1)-dimensions by Tang \cite{tang2006decay}. The further remarkable work in higher dimension was obtained in \cite{greenleaf2007oscillatory}, the authors obtained the $L^2$ estimate for the oscillatory integral operators with homogeneous polynomials satisfying various genericity assumptions.

Inspired by the method used in \cite{yang2004sharp} and \cite{shi2016sharp}, we prove our main result by embeding $T_\lambda$ into a family of analytic operators and using complex interpolation. This method requires us to establish the $L^2-L^2$ decay estimate as well as $H^1-L^1$ boundedness of operators with different amplitude functions. Before we state our main theorem, some definitions should be illustrated.\\

\begin{definition}(\cite{greenleaf2007oscillatory})
If $S(x,y)\in C^{\omega}(\mathbb{R}^{n_X}\times \mathbb{R}^{n_Y})$ with Taylor series $\sum_{\alpha,\beta}a_{\alpha,\beta}x^{\alpha}y^{\beta}$ having no pure $x$- or $z$-term, we denote the \textit{reduced Newton polyhedron} by
\[\mathcal{N}_0(S)= \text{convex hull}\l(\bigcup_{a_{\alpha,\beta}\neq 0}(\alpha,\beta)+\mathbb{R}_{+}^{n_X+n_Y}\r).\]
Then the \textit{Newton polytope} of $S(x,y)$  (at $(0,0)$) is
\[\mathcal{N}(S):=\partial(\mathcal{N}_0(S)),\]
and the \textit{Newton distance} $\delta(S)$ of $S$ is then
\[\delta(S):=\inf\{\delta^{-1}>0:(\delta^{-1},\cdots,\delta^{-1})\in \mathcal{N}(S)\}.\]
\end{definition}
These definitions correspond to  the 1-dimension definitions in\cite{phong1997newton}. In our main theorem, the next definition is necessary.\\

\begin{definition}
Denote the Hilbert-Schmidt norm of a matrix $A=(a_{ij})$ by
\begin{align}
\norm{A}_{HS}&=(\mathrm{tr}(A\cdot A^T))^{1/2} \notag \\
&=(\sum_{i,j}\abs{a_{ij}}^2)^{1/2}.
\end{align}
\end{definition}

Denote by $S^d(\R\times\R)$ the space of homogeneous polynomials of degree $d$ on $\R\times \R$. In fact, for oscillatory operators with homogeneous polynomial phases, we are only interested in polynomial phase functions not containing pure $x-$ or $y-$terms since these leave the operator norm unchanged. Thus, we denote the space consisting of such polynomials by $\mathcal{O}^d(\R\times\R)$.

Now, we formulate our main result:
\begin{thA}\label{main-result}
Suppose $S(x,y)\in \mathcal{O}^d(\R\times\R)$ and $d>2n\geq 4$, if $\norm{S_{xy}^{''}}_{HS}^{1/(d-2)}$ is a norm of $\R\times \R$, then it follows
\begin{equation}
\norm{T_\lambda}_p\lesssim \begin{cases}
\lambda^{-\delta/2} & d/(d-n)<p<d/n,\\
\lambda^{-\delta/2}\l(\log(\lambda)\r)^\delta & p=d/n \  \text{or} \  p=d/(d-n),\\
\lambda^{-1/{p'}} & 1<p<d/(d-n),\\
\lambda^{-1/p} & d/n<p<\infty.
\end{cases}
\end{equation}
where $\delta$ is the Newton distance. If $d/(d-n)<p<d/n$, the decay is sharp. If $p=d/n$ or $d/(d-n)$, the decay is sharp except possibly for a $\log(\lambda)$ term. And $d/(d-n)\leq p\leq d/n$ is not necessary to guarantee the sharp decay.
\end{thA}
To clarify the relation between \textit{Newton distance} and the $L^p$ decay rate, a proposition in \cite{greenleaf2007oscillatory} should be mentioned.
\begin{prop}\label{prop1.1}
If $S(x,y)\in \mathcal{O}^d(\R\times\R)$ satisfies the rank one condition
\[{\rm rank}(S_{xy}^{''})\geq 1,\ \text{for all}\  (x,y)\neq (0,0),\]
then $\delta(S)=2n/d$.
\end{prop}
Obviously $S(x,y)$ satisfies the \textit{rank one} condition because of the assumptions in Theorem \ref{main-result}, thus the \textit{Newton distance} in Theorem A is actually $2n/d$.\\

The main tool in our proof is the interpolation of analytic families of operators which was due to Stein \cite{stein1959extension}. Here, the analytic families of operators are
\begin{equation}\label{damped-oio}
T_\lambda^z(f)(x)=\int_{\R} e^{i\lambda S(x,y)}\norm{S_{xy}^{''}}_{HS}^z\psi(x,y)f(y)\diff{y}, \ z=\sigma+it \in\mathbb{C}.
\end{equation}
Especially, $T_\lambda^0=T_\lambda$. Theorem A naturally follows from the interpolation between $L^2-L^2$ decay of $T_\lambda^z$ and the $H^1-L^1$ mapping property of $T_\lambda^z$ as well as dual arguments.

The layout of the paper is as follows. In the next section, we give the $L^2-L^2$ decay estimate of $T_\lambda^z$. Section 3 is devoted to prove the $H^1-L^1$ boundedness of $T_\lambda^z$. In the last section, the optimality of decay will be obtained and we will also give an example to demonstrate our main theorem and a counterexample to show that $d/(d-n)\leq p\leq d/n$ is not necessary to guarantee the sharp decay.\\
\begin{notation}
In this paper, the notation $C$ denotes the positive constant used in the usual way and it may vary according to different conditions.
\end{notation}

\section{$L^2$ decay of the damped oscillatory integral operators}

In this part, the desired result follows:
\begin{theorem}\label{L2-L2}
Set $\sigma_1=-n/(d-2), \sigma_2=(d-2n)/(2(d-2))$ and consider the operators defined in \eqref{damped-oio}, if the Hessian of its phase function satisfies the rank one condition, the next estimates hold
\begin{equation}\label{L2-L2-eq}
\norm{T_\lambda^{z}}_{2}\lesssim \begin{cases}
C_z\abs{\lambda}^{-1/2}, & \sigma>\sigma_2;\\
C_z\abs{\lambda}^{-1/2}\log(\lambda), & \sigma=\sigma_2;\\
C_z\abs{\lambda}^{-[(d-2)\sigma+n]/d}, & \sigma_1<\sigma<\sigma_2.
\end{cases}
\end{equation}
\end{theorem}
\noindent Our proof roughly follows the pattern appeared in \cite{tang2006decay} and \cite{greenleaf2007oscillatory} in which the authors offered a nice viewpoint of higher dimensional oscillatory integral operators. They combined the dyadic decomposition of the entire space and the local H\"{o}rmander lemma \cite{hormander1973oscillatory} on the dyadic shell to give the next lemma.
\begin{lemma}[\cite{greenleaf2007oscillatory}]
For a homogeneous phase function $S(x,y)$ of degree $d$ with $S_{xy}^{''}$ satisfying the rank one condition
\[{\rm rank}(S_{xy}^{''})\geq 1,\ \text{for all}\  (x,y)\neq (0,0)\]
on $\mathbb{R}^{n_X}\times \mathbb{R}^{n_Y} (n_X\geq n_Y \geq 2)$, there hold
\begin{equation}
\norm{T_\lambda}_2 \leq
\begin{cases}
C\lambda^{-(n_X+n_Y)/(2d)} & \text{if}\  d>n_{X}+n_{Y},\\
C\lambda^{-1/2}\log{\lambda} & \text{if}\  d=n_{X}+n_{Y},\\
C\lambda^{-1/2} & \text{if} \  2\leq d<n_{X}+n_{Y}.
\end{cases}
\end{equation}
\end{lemma}

\noindent The difference in our proof is that we combine the dyadic decomposition and the local oscillatory estimate (Lemma 1.1 in \cite{phong1994models}). Now we turn to our proof of Theorem \ref{L2-L2}.
\begin{proof}
Since the support of $\psi(x,y)$ is compact, we my assume that $\supp(\psi)$ is contained in  $\{(x,y):\abs{(x,y)}\leq 1\}$. Considering the compactness of the sphere $\abs{(x,y)}=1$, we can make a partition of unity over the unit sphere, and then extend it to a partition of unity on $\mathbb{R}^{2n}\setminus \{0\}$, homogeneous of degree 0. Thus, to conclude the result of \eqref{L2-L2-eq}, it suffices to show that for each point on the unit sphere of $\R\times \R$, an operator supported in one of its (small enough) convex conic neighborhood has the desired decay rate. Decompose the unit ball by dyadic partition of unity $\{a_k\}$, $\sum_{k=0}^{\infty}a_k(x,y)\equiv 1$, and
\[\supp(a_k)\subseteq \{2^{-k-1}< \abs{(x,y)}\leq 2^{-k+1}\}.\]
Set $\psi_{k}=\psi a_k$ and $T_{\lambda,k}^{z}f(x)=\int_{\R} e^{i\lambda S(x,y)}\psi_k(x,y)f(y)\diff{y}$. Since the Hessian $S_{xy}^{''}$ satisfies the rank one condition, then for each $(x_0,y_0)\in \mathbb{S}^{2n-1}$, there exists at least a pair of indices $(i_0,j_0)$ such that $S_{x_{i_0}y_{j_0}}^{''}(x_0,y_0)\neq 0$. Set $C_0=\max\{|S_{x_{i}y_{j}}^{''}(x_0,y_0)|:1\leq i,j\leq n\}$, thus $C_0>0$ for each $(x_0,y_0)\in \mathbb{S}^{2n-1}$. Without confusion, we may assume $|S_{x_{1}y_{1}}^{''}(x_0,y_0)|=C_0$ and there must exist a sufficiently small neighborhood $\mathcal{U}$ of $(x_0,y_0)$ on the unit sphere such that
\begin{equation*}
C_0/2<|S_{x_{1}y_{1}}^{''}(x,y)|<2C_0,~  |S_{x_{i}y_{j}}^{''}(x,y)|<2C_0,~\forall (i,j)\neq (1,1), ~ \forall (x,y)\in \mathcal{U}.
\end{equation*}
Denote the conic convex hull of origin and $\mathcal{U}$ by $\mathcal{U}_c$. A finite number of such $\mathcal{U}_c$ cover the unit ball. Thus $\psi(x,y)$ can be assumed to be supported in $\mathcal{U}_c$. Obviously, on the support of $\psi_k$, we have $|S_{x_{1}y_{1}}^{''}(x,y)|\approx 2^{-(d-2)k}C_0$.

Writing $x=(x_1,x'),~ y=(y_1,y'),~ \phi_k^z(x,y)=\norm{S_{xy}^{''}}_{HS}^z\psi_k(x,y)$, it follows
\begin{align*}
T_{\lambda,k}^z(f)(x)&=\int_{\R} e^{i\lambda S(x,y)}\norm{S_{xy}^{''}}_{HS}^z\psi_k(x,y)f(y)\diff{y}\\
&=\int_{\mathbb{R}^{n-1}}\int_{\mathbb{R}}e^{i\lambda S(x_1,x',y_1,y')}\phi_k(x_1,x',y_1,y')f(y_1,y')\diff{y_1}\diff{y'}
\end{align*}
Set $S_{x',y'}(x_1,y_1)=S(x_1,x',y_1,y'), ~\phi_{k,x',y'}^z(x_1,y_1)=\phi_k^z(x_1,x',y_1,y')$ as well as $f_{y'}(y_1)=f(y_1,y')$, then
\begin{align*}
T_{\lambda,k}^z(f)(x_1,x')&~=\int_{\mathbb{R}^{n-1}}\int_{\mathbb{R}}e^{i\lambda S_{x',y'}(x_1,y_1)}\phi_{k,x',y'}^z(x_1,y_1)f_{y'}(y_1)\diff{y_1}\diff{y'}\\
&:=\int_{\mathbb{R}^{n-1}}\tilde{T}_{\lambda,k,x',y'}^zf_{y_1}(x_1)\diff{y'},
\end{align*}
where $\tilde{T}_{\lambda,k,x',y'}^zf_{y_1}(x_1)$ are the one dimensional oscillatory integral operators investigated in \cite{phong1994models}. Repeating the proof of Lemma 1.1 in \cite{phong1994models} and provided that $S_{x',y'}(x_1,y_1)$ is uniformly polynomial-like in $y_1$, we obtain
\begin{equation*}
\norm{\tilde{T}_{\lambda,k,x',y'}^zf_{y_1}(x_1)}_{L^2(\mathbb{R})}\lesssim \abs{z(z-1)}2^{-(d-2)k\sigma}\abs{\lambda 2^{-(d-2)k}}^{-1/2}\norm{f(\cdot,y')}_{L^2(\mathbb{R})}.
\end{equation*}
Combing this with the size of the support in $x'$ yields
\begin{align}\label{os-es}
\norm{T_{\lambda,k}^z}_2 &\lesssim \abs{z(z-1)}2^{-(d-2)k\sigma}\abs{\lambda 2^{-(d-2)k}}^{-1/2}2^{-(2n-2)k/2} \notag\\
& =C_z2^{-(d-2)k\sigma}2^{(d-2n)k/2}\abs{\lambda}^{-1/2}
\end{align}
where $C_z=\abs{z(z-1)}$.

On the other hand, from the size estimate, it is easy to verify
\begin{equation}\label{siz-es}
\norm{T_{\lambda,k}^z}_2\lesssim 2^{-(d-2)k\sigma}2^{-nk}.
\end{equation}
The estimates in \eqref{os-es} and \eqref{siz-es} are comparable if and only if
\begin{equation*}
2^{-(d-2)k\sigma}2^{k(d-2n)/2}\abs{\lambda}^{-1/2}\sim 2^{-(d-2)k\sigma}2^{-nk}, ~\text{or}~ 2^k\sim \abs{\lambda}^{1/d}
\end{equation*}
Thus
\begin{align}\label{2-sum}
\norm{T_{\lambda}^z}_2&\lesssim C_z\sum_{k=0}^{+\infty}\min\{2^{-(d-2)k\sigma}2^{k(d-2n)/2}\abs{\lambda}^{-1/2},2^{-(d-2)k\sigma}2^{-nk}\}\notag\\
&=C_z\left[\sum_{k=0}^{\frac{1}{d}\log_2|\lambda|}2^{-(d-2)k\sigma}2^{(d-2n)k/2}\abs{\lambda}^{-1/2}+
\sum_{k=\frac{1}{d}\log_2|\lambda|}^{+\infty}2^{-(d-2)k\sigma}2^{-nk}\right]\notag\\
&=C_z\left[\sum_{k=0}^{\frac{1}{d}\log_2|\lambda|}2^{{((d-2n)-2(d-2)\sigma)}k/2}\abs{\lambda}^{-1/2}+
\sum_{k=\frac{1}{d}\log_2|\lambda|}^{+\infty}2^{-k((d-2)\sigma+n)}\right].
\end{align}
If $\sigma>\sigma_2$, then $(d-2n)-2(d-2)\sigma<0$, the first sum in \eqref{2-sum} is therefore less than $C_z\abs{\lambda}^{-1/2}$, the second one is less than $C_z\abs{\lambda}^{-1/2}$.\\
If $\sigma=\sigma_2$, then $(d-2n)-2(d-2)\sigma=0$, the first term is less than $C_z|\lambda|^{-1/2}\log_2\abs{\lambda}$, the second term is less than $C_z|\lambda|^{-1/2}$.\\
If $\sigma_1<\sigma<\sigma_2$, then $(d-2n)-2(d-2)\sigma>0$ and $(d-2)\sigma+n>0$, the first sum is less than $C_z\abs{\lambda}^{-((d-2)\sigma+n)/d}$ and so is the second sum.\\
Summing up the three cases above, we complete the proof of Theorem \ref{L2-L2}.
\end{proof}

\section{$H^1-L^1$ mapping property of the damped oscillatory integral operators}
By using the result in \cite{ricci1987harmonic}, Pan \cite{pan1991hardy} establish the $H_{E}^1-L^1$ boundedness for oscillatory singular integral operators, where $H_{E}^1$ is a modified Hardy space. Later, Yang \cite{yang2004sharp} and Shi \cite{shi2016sharp} developed the method of Pan to get their corresponding $H^1-L^1$ and $H_{E}^1-L^1$ boundedness results for the oscillatory operators with homogeneous polynomial phase function. In fact, based on these works, the next result can be obtained.
\begin{theorem}\label{H1-L1}
Define an operator
\begin{equation*}
T^Pf(x)=\int_{\R} e^{iP(x,y)}\norm{S_{xy}^{''}}_{HS}^{\sigma_1+it}\psi(x,y)f(y)\diff{y}
\end{equation*}
 where $P(x,y)=\sum_{\alpha,\beta}c_{\alpha,\beta}x^{\alpha}y^{\beta}$ is a higher dimensional polynomial with $c_{0,\beta}=0$ for any $\beta$.
If $\norm{S_{xy}^{''}}_{HS}^{1/(d-2)}$ is a norm of $\R\times \R$, then $T^P$ maps $H^1(\R)$ to $L^1(\R)$ with operator norm less than $C(1+\abs{t})$ in which $C$ is a constant independent of the coefficients of $P(x,y)$.
\end{theorem}
The inductive argument in Pan \cite{pan1991hardy} starts with the the $L^p$ boundedness of the oscillatory singular integral operator obtained in \cite{ricci1987harmonic}. This method requires us to consider the following operator
\begin{equation*}
T_0(f)(x)=\int_{\R} \norm{S_{xy}^{''}}_{HS}^{\sigma_1+it}\psi(x,y)f(y)\diff{y}.
\end{equation*}
If we set $K(x,y)=\norm{S_{xy}^{''}}_{HS}^{\sigma_1+it}\psi(x,y)$, then the operator equals
\begin{equation}\label{oio-wn-o}
T_0(f)(x)=\int_{\R} K(x,y)f(y)\diff{y}
\end{equation}
\subsection{Mapping property of $T_0$.}
\begin{theorem}
Considering the operator $T_0$ defined in \eqref{oio-wn-o}, if $\norm{S_{xy}^{''}}_{HS}^{1/(d-2)}$ is a norm of $\R\times \R$, it follows
\begin{enumerate}
\item[{\rm (i)}] $T_0$ is of type $(p,p)$ whenever $1<p<+\infty$;
\item[{\rm (ii)}] $T_0$ is of weak type $(1,1)$;
\item[{\rm (iii)}] $T_0$ maps $H^1(\R)$ to $L^1(\R)$ with operator norm less than $C(1+|t|)$.
\end{enumerate}
\end{theorem}
\begin{proof}
(i) By the assumption that $\norm{S_{xy}^{''}}_{HS}^{1/(d-2)}$ is a norm of $\R\times \R$, and the fact that the norms in finite dimensional linear normed space are equivalent, we have $\norm{S_{xy}^{''}}_{HS}^{1/(d-2)}\approx (|x|^2+|y|^2)^{1/2}\approx |x|+|y|$. Since $K(x,y)=\norm{S_{xy}^{''}}_{HS}^{\sigma_1+it}\psi(x,y)$ and $\psi\in C_0^\infty(\R\times\R)$, then
\begin{equation*}
\abs{K(x,y)}\lesssim \frac{1}{(\abs{x}+\abs{y})^n}\approx \frac{1}{\abs{x}^n+\abs{y}^n}.
\end{equation*}
For any $f\in L^p(\R), ~1<p<+\infty$,
\begin{align*}
\norm{T_0f(x)}_p =\left(\int\abs{T_0f(x)}^p\diff{x}\right)^{1/p}
&\leq\left( \int\abs{\int |K(x,y)||f(y)|\diff{y}}^p\diff{x}\right)^{1/p}\\
&\lesssim\left(\int\abs{\int \frac{|f(y)|}{\abs{x}^n+\abs{y}^n}\diff{y}}^p\diff{x}\right)^{1/p}\\
&=\left(\int \abs{\int \frac{|f(|x|y)|}{1+|y|^n}\diff{y}}^p\diff{x}\right)^{1/p}
\end{align*}
By using the polar coordinate $x=R\theta, ~y=r\omega$, the last term equals
\begin{align*}
&\omega_{n-1}^{1/p}\left(\int_0^{+\infty} \abs{\int_0^{+\infty} \int_{\mathbb{S}^{n-1}} \frac{|f(Rr\omega)|}{1+|r|^n}r^{n-1}\diff{\omega}\diff{r}}^pR^{n-1}\diff{R}\right)^{1/p}\\
&\leq \omega_{n-1}^{1/p} \int_0^{+\infty} \left(\int_0^{+\infty} \abs{\int_{\mathbb{S}^{n-1}} |f(Rr\omega)|\diff{\omega}}^p R^{n-1}\diff{R}\right)^{1/p}\frac{r^{n-1}}{1+|r|^n}\diff{r}\\
&\leq \omega_{n-1}^{1/p} \int_0^{+\infty} \left(\int_0^{+\infty} \int_{\mathbb{S}^{n-1}} |f(Rr\omega)|^p R^{n-1}\diff{\omega}\cdot \omega_{n-1}^{p/q} \diff{R}\right)^{1/p}\frac{r^{n-1}}{1+|r|^n}\diff{r}\\
&=\omega_{n-1}\int_0^{+\infty} \left(\int_0^{+\infty} \int_{\mathbb{S}^{n-1}} |f(Rr\omega)|^p R^{n-1}\diff{\omega}\diff{R}\right)^{1/p}\frac{r^{n-1}}{1+|r|^n}\diff{r}\\
&=\omega_{n-1}\int_0^{+\infty} \left(\int_0^{+\infty} \int_{\mathbb{S}^{n-1}} |f(R\omega)|^p R^{n-1}\diff{\omega}\diff{R}\right)^{1/p}\frac{r^{-n/p}\cdot r^{n-1}}{1+|r|^n}\diff{r}\\
&=\omega_{n-1}\norm{f}_p\int_0^{+\infty}\frac{r^{-n/p}\cdot r^{n-1}}{1+|r|^n}\diff{r}.
\end{align*}
Since the integral in the last term is finite, then $\norm{T_0f}_p\leq C\norm{f}_p$ obviously.\\
(ii) For any $f\in L^1(\R)$ and $\lambda>0$, we can decompose it into $f(x)=g(x)+b(x)$ by Calder\`{o}n-Zygmund decomposition. Here
\begin{align*}
&b=\sum b_j;\\
&b_j=(f(x)-\frac{1}{|Q_j|}\int_{Q_j}f(y)\diff{y})\chi_{Q_j}(x)
\end{align*}
where $Q_j$ is a cube centered at $x_{Q_j}$ with side length $d_{Q_j}$. Let $Q_j^*$ denote the cube centered at $x_{Q_j}$ with side length $Md_{Q_j}$ where $M$ is a constant large enough. Thus
\begin{align*}
|\{x: |T_0f(x)|>\lambda\}|&\leq |\{x: |T_0g(x)|>\lambda/2\}|+|\{x: |T_0b(x)|>\lambda/2\}|\\
&\leq 2/\lambda\norm{g}_1+\sum_j|Q_j^*|+|\{x\in (\cup_j Q_j^*)^c: |T_0b(x)|>\lambda/2\}|\\
&\lesssim\norm{g}_1/\lambda+\norm{f_1}_1/\lambda+|\{x\in (\cup_j Q_j^*)^c: |T_0b(x)|>\lambda/2\}|\\
&\lesssim\norm{f}_1/\lambda+|\{x\in (\cup_j Q_j^*)^c: |T_0b(x)|>\lambda/2\}|
\end{align*}
From \eqref{oio-wn-o}, it follows
\begin{align*}
|\{x\in (\cup_j Q_j^*)^c: |T_0b(x)|>\lambda/2\}|&\leq \frac{2}{\lambda}\int_{(\cup_j Q_j^*)^c} |T_0b(x)|\diff{x}\\
&=\frac{2}{\lambda}\int_{(\cup_j Q_j^*)^c} |\int_{\R} K(x,y)b(y)\diff{y}|\diff{x}\\
&=\frac{2}{\lambda}\int_{(\cup_j Q_j^*)^c} |\sum_j\int_{Q_j} K(x,y)b_j(y)\diff{y}|\diff{x}\\
&\leq \sum_j\frac{2}{\lambda}\int_{(Q_j^*)^c} |\int_{Q_j} K(x,y)b_j(y)\diff{y}|\diff{x}.
\end{align*}
By the vanishing property of $b_j$, we have
\begin{align*}
&\int_{(Q_j^*)^c} |\int_{Q_j} K(x,y)b_j(y)\diff{y}|\diff{x}\\
&=\int_{(Q_j^*)^c} |\int_{Q_j} (K(x,y)-K(x,x_Q))b_j(y)\diff{y}|\diff{x}.\\
\end{align*}
Obviously,
\begin{align*}
&\int_{(Q_j^*)^c} |\int_{Q_j} (K(x,y)-K(x,x_{Q_j}))b_j(y)\diff{y}|\diff{x}\\
&\leq \sup_{y\in Q_j}\int_{(Q_j^*)^c} |K(x,y)-K(x,x_{Q_j})|\diff{x}\cdot\int_{Q_j}|b_j(y)|\diff{y}.
\end{align*}
If we can prove that
\begin{equation}\label{Hormander-cond}
\sup_{y\in Q_j}\int_{(Q_j^*)^c} |K(x,y)-K(x,x_{Q_j})|\diff{x}\leq C
\end{equation}
 where $C$ is a constant independent of $Q_j$, on account of $\sum_j\norm{b_j}_1\leq C\norm{f}_1$, we will conclude (ii). However, analysis of this supremum should be split into two cases as follow.\\

Case I: $|x_{Q_j}|<2d_{Q_j}$.\\
In this case, $|y-x_{Q_j}|<d_{Q_j}$ yields $|y|<3d_{Q_j}$. Note that each entry in $S_{xy}^{''}$ is a homogeneous polynomial of degree $d-2$, and $|x|\approx|x-x_{Q_j}|>Md_{Q_j}\gg |y|$. Thus provided that $|\sigma_1|\leq 1/2$ we have
\begin{equation*}
|\nabla_y K(x,y)|\leq \frac{C(1+|t|)}{|x|^{n+1}}.
\end{equation*}
Therefore
\begin{align*}
\sup_{y\in Q_j}\int_{(Q_j^*)^c} |K(x,y)-K(x,x_Q)|\diff{x}&\lesssim\int_{|x-x_{Q_j}|>Md_{Q_j}} \frac{d_{Q_j}}{|x|^{n+1}}\diff{x}\\
&\leq\int_{|x|>(M-2)d_{Q_j}} \frac{d_{Q_j}}{|x|^{n+1}}\diff{x}\\
&\leq C.
\end{align*}

Case II: $|x_{Q_j}|\geq 2d_{Q_j}$.\\
In this case, since $y\in Q_j$ then $|y-x_{Q_j}|<d_{Q_j}$, i.e. $\frac{1}{2}|x_{Q_j}|<|y|<\frac{3}{2}|x_{Q_j}|$. Hence
\begin{align*}
&\sup_{y\in Q_j}\int_{(Q_j^*)^c} |K(x,y)-K(x,x_Q)|\diff{x}\\
&\lesssim\int_{|x-x_{Q_j}|>Md_{Q_j}} |K(x,y)-K(x,x_Q)|\diff{x}\\
&=\int_{|x-x_{Q_j}|>M|x_{Q_j}|}\cdots\diff{x}+\int_{Md_{Q_j}<|x-x_{Q_j}|\leq M|x_{Q_j}|}\cdots\diff{x} \\
:&=A+B.
\end{align*}
 Observe that $|y|\approx |x_{Q_j}|$, the estimate of $A$ is same with Case I and we omit here. For $B$, we have
\begin{align*}
B=&\int_{Md_{Q_j}<|x-x_{Q_j}|\leq M|x_{Q_j}|}|K(x,y)-K(x,x_Q)|\diff{x}\\
\leq &\int_{Md_{Q_j}<|x-x_{Q_j}|\leq M|x_{Q_j}|}|K(x,y)|+|K(x,x_Q)|\diff{x}\\
\leq &\int_{Md_{Q_j}<|x-x_{Q_j}|\leq M|x_{Q_j}|}\frac{1}{|x|^n+|y|^n}+\frac{1}{|x|^n+|x_{Q_j}|^n}\diff{x}\\
\leq &\int_{|x-x_{Q_j}|\leq M|x_{Q_j}|}\frac{1}{|y|^n}+\frac{1}{|x_{Q_j}|^n}\diff{x}\\
\leq &|x_{Q_j}|^n(\frac{1}{|y|^n}+\frac{1}{|x_{Q_j}|^n})\\
\leq &C.
\end{align*}
Thus, the proof of (ii) is complete.\\
(iii) Let $a$ denote a $H^1$ atom associated with a cube $Q$ centered at $x_Q$ with side length $d_Q$ and
\begin{align}
\supp& a\subset Q;\label{atom-p1}\\
\norm{a}_\infty&\leq\frac{1}{|Q|};\label{atom-p2}\\
\int_{Q}a&\diff{x}=0.\label{atom-p3}
\end{align}
Our goal is to prove
\begin{equation*}
\norm{T_0a}_1\leq C,
\end{equation*}
where $C$ is independent of $Q$. Analogous to the argument of (ii), the proof should be divided into two cases.

Case I: $|x_Q|<2d_Q$.
\begin{align*}
\norm{T_0a}_1=\int \abs{T_0a}\diff{x}
&=\int_{|x-x_Q|\leq M|d_Q|}\abs{T_0a}\diff{x}+\int_{|x-x_Q|> M|d_Q|}\abs{T_0a}\diff{x}\\
:&= I_1+I_2.
\end{align*}
From the $L^p(1<p<+\infty)$ boundedness of $T_0$ in (i), we have
\begin{equation*}
I_1=\int_{|x-x_Q|\leq M|d_Q|}\abs{T_0a}\diff{x}
\leq(M|d_Q|)^{n/2}\norm{T_0a}_2
\leq(M|d_Q|)^{n/2}\norm{a}_2
\leq C.
\end{equation*}
By employing the argument of Case I in (ii) to $I_2$, we obtain
\begin{align*}
I_2=\int_{(Q^*)^c}\abs{T_0a}\diff{x}=\int_{(Q^*)^c}\abs{\int_{Q} K(x,y)a(y)\diff{y}}\diff{x}
=&\int_{(Q^*)^c}\abs{\int_{Q} (K(x,y)-K(x,x_Q))a(y)\diff{y}}\diff{x}\\
\leq & \sup_{y\in Q}\int_{(Q^*)^c} |K(x,y)-K(x,x_Q)|\diff{x}\cdot \int_Q|a(y)|\diff{y}.
\end{align*}
Thus \eqref{Hormander-cond} together with \eqref{atom-p2} implies $I_2\leq C$.

Case II: $|x_Q|\geq 2d_Q$.
\begin{align*}
\norm{T_0a}_1=\int \abs{T_0a}\diff{x}
&=\int_{|x-x_Q|\leq M|x_Q|}\abs{T_0a}\diff{x}+\int_{|x-x_Q|> M|x_Q|}\abs{T_0a}\diff{x}\\
:&= I_3+I_4.
\end{align*}
Since $|x_Q|\geq 2d_Q$ and $y\in Q$, then $\frac{1}{2}|x_Q|\leq |y|\leq \frac{3}{2}|x_Q|$, it is easy to verify
\begin{align*}
I_3=\int_{|x-x_Q|\leq M|x_Q|}\abs{T_0a}\diff{x}
&=\int_{|x-x_Q|\leq M|x_Q|}\abs{\int_QK(x,y)a(y)dy}\diff{x}\\
&\leq \int_{|x|\leq (M+1)|x_Q|}\int_Q\frac{1}{|x|^n+|y|^n}|a(y)|\diff{y}\diff{x}\\
&= \int_Q\int_{|x|\leq (M+1)|x_Q|/|y|}\frac{1}{|x|^n+1}\diff{x} |a(y)|\diff{y}\\
&= \int_Q\int_{|x|\leq 2(M+1)}\frac{1}{|x|^n+1}\diff{x} |a(y)|\diff{y}\\
&\leq C.
\end{align*}
Observe that
\begin{align*}
I_4\leq \int_{|x-x_Q|> M|d_Q|}\abs{T_0a}\diff{x}
&=\int_{(Q^*)^c}\abs{T_0a}\diff{x}\\
&=\int_{(Q^*)^c}\abs{\int_{Q} K(x,y)a(y)\diff{y}}\diff{x}\\
&=\int_{(Q^*)^c}\abs{\int_{Q} (K(x,y)-K(x,x_Q))a(y)\diff{y}}\diff{x}\\
&\leq \sup_{y\in Q}\int_{(Q^*)^c} |K(x,y)-K(x,x_Q)|\diff{x}\cdot \int_Q|a(y)|\diff{y}.
\end{align*}
On account of \eqref{Hormander-cond}, $I_4\leq C$ obviously.

\end{proof}
\subsection{Some useful lemmas}
Before we prove Theorem \ref{H1-L1}, some useful lemmas should be stated.
\begin{lemma}\label{van-der-corput}
Let $\phi(x)=\sum_{|\alpha|\leq d}a_{\alpha}x^\alpha$ be a real-valued polynomial in $\R$ of degree d, and $\varphi(x)\in C_{0}^\infty(\R)$. If $a_{\alpha_0}\neq 0$ for $\alpha_0=d$ we have
\begin{equation*}
\abs{\int_{\R}e^{i\phi(x)}\varphi(x)\diff{x}}\leq C\abs{a_{\alpha_0}}^{-1/d}(\norm{\varphi}_\infty+\norm{\nabla\varphi}_1).
\end{equation*}
\end{lemma}
More details about this lemma can be found in \cite{stein1993harmonic}. The following lemma about polynomial was due to Ricci and Stein \cite{ricci1987harmonic}.
\begin{lemma}\label{poly-es}
Let $P(x)=\sum_{\abs{\alpha}\leq d}a_\alpha x^\alpha$ denote a polynomial in $\R$ of degree d. Suppose $\epsilon<1/d$, then
\begin{equation*}
\int_{\abs{x}\leq 1}\abs{P(x)}^{-\epsilon}\diff{x}\leq A_\epsilon \left(\sum_{\abs{\alpha}\leq d}\abs{a_\alpha}\right)^{-\epsilon}.
\end{equation*}
The bound $A_\epsilon$ depends on $\epsilon$ and dimension $n$, but not on the coefficients \{$a_\alpha$\}.
\end{lemma}

\subsection{Proof of Theorem \ref{H1-L1}}
\begin{proof}
For the atoms in Hardy space, we use the same notations as the proof of (iii). To prove this theorem, we shall use induction on the degree $l$ of $y$ in $P(x,y)$.\\
If $l=0$, $P(x,y)$ contains only the pure $x$-term. Then from (iii) we know
\begin{align*}
\norm{T^Pa}_1=\norm{T_0a}_1\leq C.
\end{align*}
We suppose that $\norm{T^Pa}_1\leq C$ holds if the degree of $P$ in $y$ is less than $l$. As the proof of (iii), we consider two cases:

Case I: $|x_Q|\leq 2d_Q$.\\
\begin{align*}
\int \abs{T^Pa(x)}\diff{x}
&=\int_{\abs{x-x_Q}\leq Md_Q}\abs{T^Pa(x)}\diff{x}+\int_{\abs{x-x_Q}> Md_Q}\abs{T^Pa(x)}\diff{x}\\
:&=I_5+I_6.
\end{align*}
Taking absolute value in $I_5$ and recalling the argument of (i), $I_5\leq C$ obviously.\\
Write $P(x,y)=\sum_{|\alpha|\geq 1,|\beta|=l}c_{\alpha,\beta}x^{\alpha}y^{\beta}+Q(x,y)$, where $Q(x,y)$ is a polynomial with degree in $y$ less than or equal to $l-1$. We spilt $I_6$ into two parts,
\begin{align*}
I_6=&\int_{Md_Q<\abs{x-x_Q}<r}\abs{\int e^{iP(x,y)}K(x,y)a(y)\diff{y}}\diff{x}+\\
&\int_{\abs{x-x_Q}\geq\max\{Md_Q, r\}}\abs{\int e^{iP(x,y)}K(x,y)a(y)\diff{y}}\diff{x}\\
:&=I_7+I_8.
\end{align*}
Then for $I_7$, there is
\begin{align*}
I_7&=\int_{Md_Q<\abs{x-x_Q}<r}\abs{T^Pa}\diff{x}\\
&=\int_{Md_Q<\abs{x-x_Q}<r}\abs{\int e^{iP(x,y)}K(x,y)a(y)\diff{y}}\diff{x}\\
&=\int_{Md_Q<\abs{x-x_Q}<r}\abs{\int (e^{iP(x,y)}-e^{iQ(x,y)})K(x,y)a(y)\diff{y}}\diff{x}+\\
&\int_{Md_Q<\abs{x-x_Q}<r}\abs{\int e^{iQ(x,y)}K(x,y)a(y)\diff{y}}\diff{x}\\
:&=I_9+I_{10}.
\end{align*}
For $I_9$, we have
\begin{align*}
I_9=&\int_{Md_Q<\abs{x-x_Q}<r}\abs{\int (e^{iP(x,y)}-e^{iQ(x,y)})K(x,y)a(y)\diff{x}}\diff{x}\\
\lesssim& \int_{Md_Q<\abs{x-x_Q}<r}\int_{Q}|\sum_{|\alpha|\geq 1,|\beta|=l}c_{\alpha,\beta}x^{\alpha}y^{\beta}|\frac{1}{|x|^n+|y|^n}|a(y)|\diff{y}\diff{x}\\
\leq& \int_{Md_Q<\abs{x-x_Q}<r}\int_{Q}|\sum_{|\alpha|\geq 1,|\beta|=l}c_{\alpha,\beta}x^{\alpha}y^{\beta}|\frac{1}{|x|^n}|a(y)|\diff{y}\diff{x}\\
\leq& \int_{Md_Q<\abs{x-x_Q}<r}\int_{Q}\sum_{|\alpha|\geq 1,|\beta|=l}|c_{\alpha,\beta}||x|^{|\alpha|}|y|^l\frac{1}{|x|^n}|a(y)|\diff{y}\diff{x}\\
\leq&\int_{Md_Q<\abs{x-x_Q}<r}\int_{Q}\sum_{|\alpha|\geq 1,|\beta|=l}|c_{\alpha,\beta}||x|^{|\alpha|-n}|y|^l|a(y)|\diff{y}\diff{x}\\
\lesssim &\int_{Md_Q<\abs{x-x_Q}<r}\int_{Q}\sum_{|\alpha|\geq 1,|\beta|=l}|c_{\alpha,\beta}||x|^{|\alpha|-n}|d_Q|^l|a(y)|\diff{y}\diff{x}
\end{align*}
Since $|x_Q|<2d_Q$ and $\norm{a}_\infty\leq \frac{1}{|Q|}$, then
\begin{align*}
&\int_{Md_Q<\abs{x-x_Q}<r}\int_{Q}\sum_{|\alpha|\geq 1,|\beta|=l}|c_{\alpha,\beta}||x|^{|\alpha|-n}|d_Q|^l|a(y)|\diff{y}\diff{x}\\
&\leq \int_{|x|\leq 2r}\sum_{|\alpha|\geq 1,|\beta|=l}|c_{\alpha,\beta}||x|^{|\alpha|-n}|d_Q|^l\diff{x}\\
&\lesssim |d_Q|^l\sum_{|\alpha|\geq 1,|\beta|=l}|c_{\alpha,\beta}||r|^{|\alpha|}
\end{align*}
There must exist $(\alpha_0,\beta_0)$ such that $|\alpha_0|\geq 1,|\beta_0|=l$, and
\begin{equation*}
|d_Q|^{l/|\alpha_0|}|c_{\alpha_0,\beta_0}|^{1/|\alpha_0|}=\max_{|\alpha|\geq 1,|\beta|=l}|d_Q|^{l/|\alpha|}|c_{\alpha,\beta}|^{1/|\alpha|}.
\end{equation*}
Set $r^{-1}=|d_Q|^{l/|\alpha_0|}|c_{\alpha_0,\beta_0}|^{1/|\alpha_0|}$. Then $I_9\leq C$ obviously. On the other hand, by inductive hypothesis, $I_{10}\leq C$. Thus we complete the argument of $I_7$. For $I_8$ we have
\begin{align*}
I_8=&\int_{\abs{x-x_Q}\geq \max\{Md_Q, r\}}\abs{\int e^{iP(x,y)}K(x,y)a(y)\diff{y}}\diff{x}\\
\leq&\int_{\abs{x-x_Q}\geq \max\{Md_Q, r\}}\abs{\int e^{iP(x,y)}(K(x,y)-K(x,x_Q))a(y)\diff{y}}\diff{x}\\
&+\int_{\abs{x-x_Q}\geq \max\{Md_Q, r\}}\abs{K(x,x_Q)}\abs{\int e^{iP(x,y)}a(y)\diff{y}}\diff{x}\\
:&=I_{11}+I_{12}.
\end{align*}
From \eqref{Hormander-cond}, it is easy to verify $I_{11}\leq C$. Given $|x_Q|\leq 2d_Q$, we have $|K(x,x_Q)|\lesssim \frac{1}{|x|^n+|x_Q|^n}\approx \frac{1}{|x-x_Q|^n}$, therefore
\begin{align*}
I_{12}\leq&\int_{\abs{x-x_Q}\geq r}\abs{K(x,x_Q)}\abs{\int e^{iP(x,y)}a(y)\diff{y}}\diff{x}\\
\lesssim&\int_{\abs{x-x_Q}\geq r}\frac{1}{|x-x_Q|^n}\abs{\int e^{iP(x,y)}a(y)\diff{y}}\diff{x}\\
=&\sum_{j=0}^{+\infty}\int_{R_j}\frac{1}{|x-x_Q|^n}\abs{\int e^{iP(x,y)}a(y)\diff{y}}\diff{x}\\
=&\sum_{j=0}^{+\infty}\int_{R_j}\frac{1}{|x-x_Q|^n}\abs{\chi_{R_j}(x)\int e^{iP(x,y)}a(y)\diff{y}}\diff{x}\\
\lesssim&\sum_{j=0}^{+\infty}\int_{R_j}\frac{1}{2^{jn}r^n}\abs{\chi_{R_j}(x)\int e^{iP(x,y)}a(y)\diff{y}}\diff{x}
\end{align*}
where $R_j=\{x\in\R:2^jr\leq|x-x_Q|<2^{j+1}r\}$. Set $x=x_Q+2^jru, y=x_Q+d_Qv$ and $P_j(u,v)=P(x_Q+2^jru,x_Q+d_Qv)$, then
\begin{align*}
I_{12}\lesssim &\sum_{j=0}^{+\infty}\int_{2^jr\leq|x-x_Q|<2^{j+1}r}\frac{1}{2^{jn}r^n}\abs{\chi_{R_j}(x)\int e^{iP(x,y)}a(y)\diff{y}}\diff{x}\\
=&\sum_{j=0}^{+\infty}\int_{1\leq|u|<2}\abs{\hat{\chi}_{R_j}(u)\int e^{iP_j(u,v)}d_Q^na(x_Q+d_Qv)\diff{v}}\diff{u}.
\end{align*}
Suppose $\varphi\in C_0^\infty(\R)$ and
\[\varphi(v)\equiv 1 ~~~\text{for}~~~|v|\leq 1,~~~~~\varphi(v)\equiv 0 ~~~\text{for}~~~|v|\geq 2.\]
Define an operator $L_j$ by
\begin{equation*}
L_jf(u)=\hat{\chi}_{R_j}(u)\int e^{iP_j(u,v)}\varphi(v)f(v)\diff{v}.
\end{equation*}
Then
\begin{equation}\label{operator-sum}
I_{12}\lesssim\sum_{j=0}^{+\infty}\int_{1\leq|u|<2}|L_j(b)(u)|\diff{u}
\end{equation}
where $b(v)=d_Q^na(x_Q+d_Qv)$ is an atom associated with the unit cube centered at the origin. Set
\begin{align*}
L_j(u,w)={\rm{Ker}}(L_jL_j^*)=\hat{\chi}_{R_j}(u)\hat{\chi}_{R_j}(w)\int e^{iP_j(u,v)-iP_j(w,v)}|\varphi(v)|^2\diff{v}.
\end{align*}
Since
\begin{align*}
&P_j(u,v)-P_j(w,v)\\
&=\sum_{|\alpha|\geq 1,|\beta|=l}c_{\alpha,\beta}[(x_Q+2^jru)^{\alpha}-(x_Q+2^jrw)^{\alpha}](x_Q+d_Qv)^{\beta}+\tilde{Q}(u,w,v).
\end{align*}
Then from Lemma \ref{van-der-corput} we have
\begin{align*}
|L_j(u,w)|&\leq \hat{\chi}_{R_j}(u)\hat{\chi}_{R_j}(w)\abs{\sum_{|\alpha|\geq 1} c_{\alpha,\beta_0}[(x_Q+2^jru)^{\alpha}-(x_Q+2^jrw)^{\alpha}]d_Q^l}^
{-1/l}\\
&=\hat{\chi}_{R_j}(u)\hat{\chi}_{R_j}(w)\abs{\sum_{|\alpha|\geq 1}
c_{\alpha,\beta_0}\abs{2^jr}^{\abs{\alpha}}\left[\left(\frac{x_Q}{2^jr}+u\right)^{\alpha}-\left(\frac{x_Q}{2^jr}+w\right)^{\alpha}\right]d_Q^l}^{-1/l}\\
&=\hat{\chi}_{R_j}(u)\hat{\chi}_{R_j}(w)\abs{\sum_{|\alpha|\geq 1}\frac{c_{\alpha,\beta_0}d_Q^l}{|c_{\alpha_0,\beta_0}|^{\abs{\alpha}/|\alpha_0|}|d_Q|^{\abs{\alpha}l/|\alpha_0|}}
2^{j|\alpha|}\left[\left(\frac{x_Q}{2^jr}+u\right)^{\alpha}-\left(\frac{x_Q}{2^jr}+w\right)^{\alpha}\right]}^{-1/l}\\
:&=\hat{\chi}_{R_j}(u)\hat{\chi}_{R_j}(w)\abs{\sum_{|\alpha|\geq 1}b_{\alpha,\beta_0} 2^{j|\alpha|}\left[\left(\frac{x_Q}{2^jr}+u\right)^{\alpha}-\left(\frac{x_Q}{2^jr}+w\right)^{\alpha}\right]}^{-1/l}
\end{align*}
On the other hand, it is obvious that $|L_j(u,w)|\leq C$, for a large number $N$ we have
\begin{equation}\label{kernel}
|L_j(u,w)|\leq C\hat{\chi}_{R_j}(u)\hat{\chi}_{R_j}(w)\abs{\sum_{|\alpha|\geq 1}b_{\alpha,\beta_0} 2^{j|\alpha|}\left[\left(\frac{x_Q}{2^jr}+u\right)^{\alpha}-\left(\frac{x_Q}{2^jr}+w\right)^{\alpha}\right]}^{-1/Nl}.
\end{equation}
Now we figure out the coefficient of the term $u^{\alpha_0}$ in the right hand of \eqref{kernel} and denote it by $A_{\alpha_0}$. Thus
\begin{equation*}
A_{\alpha_0}=2^{j|\alpha_0|}b_{\alpha_0,\beta_0}+\sum_{|\alpha|\geq |\alpha_0|+1}2^{j|\alpha|}C_{d,\alpha} b_{\alpha,\beta_0}\left(\frac{x_Q}{2^jr}\right)^{\alpha-\alpha_0}
\end{equation*}
where $C_{d,\alpha}$ is a constant only depending on the degree of $P$ and $\alpha$. From Lemma \ref{poly-es} we have
\begin{align*}
\sup_w\int_{\R}|L_j(u,w)|\diff{u}&\leq C(A_{\alpha_0})^{-1/Nl}.
\end{align*}
 Because $|b_{\alpha_0,\beta_0}|=1$, $|b_{\alpha,\beta_0}|\leq 1$ and $r>Md_Q>\frac{M}{2}|x_Q|$, we can choose $M$ large enough such that
\begin{equation*}
2^{j|\alpha_0|-1}\leq |A_{\alpha_0}|\leq 3\cdot2^{j|\alpha_0|-1}.
\end{equation*}
Thus we can obtain
\begin{equation*}
\sup_w\int_{\R}|L_j(u,w)|\diff{u}\leq C2^{-j\theta},
\end{equation*}
where $\theta$ ia a positive constant independent of the coefficients of $P$. The same method can be applied to $\sup_u\int_{\R}|L_j(u,w)|$ and leads to the same estimate. By Schur lemma we have
\begin{equation*}
\norm{L_j}_2\leq C2^{-j\theta}.
\end{equation*}
Now we come back to \eqref{operator-sum}, by H\"{o}lder inequality
\begin{equation*}
I_{12}\leq \sum_{j=0}^{+\infty}C\norm{L_j(b)}\leq \sum_{j=0}^{+\infty}C\norm{L_j}_2\norm{b}_2 \leq C.
\end{equation*}

Case II: $|x_Q|> 2d_Q$.\\
In this case, we decompose the integral into two parts
\begin{align*}
\int \abs{T^Pa}\diff{x}
&=\int_{\abs{x-x_Q}\leq M|x_Q|}\abs{T^Pa}\diff{x}+\int_{\abs{x-x_Q}> M|x_Q|}\abs{T^Pa}\diff{x}\\
:&=I_{13}+I_{14}.
\end{align*}
  We shall get $I_{13}\leq C$ from the analogue of $I_3$. On the other hand, $I_{14}$ is similar to $I_6$, following the same pattern to deal with $I_6$ yields $I_{14}\leq C$. Thus we complete our proof.
\end{proof}

\section{Optimality of decay rates and examples}
 The optimality of decay rates can be derived from the proof of Theorem 4.1 in \cite{greenleaf2007oscillatory} and we omit here. Next we give an example to demonstrate our main result.\\

Let $n=2,d=6$ and $S(x,y)=\frac{1}{5}(x_1^5y_1+x_1y_1^5+x_1x_2^4y_2+x_1y_1^4y_2+x_1^4x_2y_1+x_2y_1y_2^4+x_2^5y_2+x_2y_2^5)$, then the Hessian matrix of $S(x,y)$ is
\begin{eqnarray*}
S_{xy}^{''}=
\l(\begin{array}{cc}
    x_1^4+y_1^4 &x_2^4+y_1^4\\

   x_1^4+y_2^4 &x_2^4+y_2^4\\
  \end{array}
  \r)
\end{eqnarray*}
Hence
\begin{equation*}
\norm{S_{xy}^{''}}_{HS}^{1/(d-2)}=\l[(x_1^4+y_1^4)^2+(x_1^4+y_2^4)^2+(x_2^4+y_1^4)^2+(x_2^4+y_2^4)^2\r]^{1/8}.
\end{equation*}
In fact the equation above can be regarded as composition of three different simple norms. Then $\norm{S_{xy}^{''}}_{HS}^{1/(d-2)}$ is a norm in $\mathbb{R}^2\times\mathbb{R}^2$ obviously. Thus this example satisfies the decay estimate in Theorem A.

If we let $n=2,d=6$ and $S(x,y)=\frac{1}{5}(x_1^5y_1+x_1y_1^5+x_2^5y_2+x_2y_2^5)$, then the Hessian matrix of $S(x,y)$ is
\begin{eqnarray*}
S_{xy}^{''}=
\l(\begin{array}{cc}
    x_1^4+y_1^4 &0\\

   0 &x_2^4+y_2^4\\
  \end{array}
  \r).
\end{eqnarray*}
This is the most simple case because the related oscillatory integral operator can be separated variables. By iterating the one-dimensional result of \cite{shi2016sharp} we can show that
\begin{equation*}
\norm{T_\lambda}_p\leq C\lambda^{-1/3} ~\text{for}~ 6/5\leq p\leq 6.
 \end{equation*}
Thus $d/(d-n)<p<d/n$ is not necessary to guarantee the sharp decay.

\section{Acknowledgement} This work was supported by the National Natural Science Foundation of China [grant numbers 11471309, 11271162 and 11561062].


\begin{thebibliography}{10}

\bibitem{arnold1986singularites}
Arnold V~I , Varchenko A~N, and
 Gusein-Zade S~M.
\newblock {\em Singularit{\'e}s des applications diff{\'e}rentiable}.
\newblock Mir, 1986.

\bibitem{greenblatt2005sharp}
Greenblatt M.
\newblock Sharp $L^2$ estimates for one-dimensional oscillatory integral
  operators with $C^\infty$ phase.
\newblock {\em Amer J Math}, 2005, 127: 659--695.

\bibitem{greenleaf2007oscillatory}
Greenleaf A, Pramanik M, and Tang W.
\newblock Oscillatory integral operators with homogeneous polynomial phases in
  several variables.
\newblock {\em J Funct Anal}, 2007, 244: 444--487.

\bibitem{greenleaf1999oscillatory}
Greenleaf A and Seeger A.
\newblock On oscillatory integral operators with folding canonical relations.
\newblock {\em Studia Math}, 1999, 132: 125--139.

\bibitem{hormander1973oscillatory}
H{\"o}rmander L.
\newblock Oscillatory integrals and multipliers on F$L^p$.
\newblock {\em Ark Mat}, 1973, 11: 1--11.

\bibitem{pan1991hardy}
Pan Y.
\newblock Hardy spaces and oscillatory singular integrals.
\newblock {\em Rev Mat Iberoam}, 1991, 7: 55--64.

\bibitem{phong1992oscillatory}
Phong D H~ and Stein E M~.
\newblock Oscillatory integrals with polynomial phases.
\newblock {\em Invent Math}, 1992, 110: 39--62.

\bibitem{phong1994models}
Phong D H~ and Stein E M~.
\newblock Models of degenerate fourier integral operators and radon transforms.
\newblock {\em Ann of Math (2)}, 1994, 140: 703--722.

\bibitem{phong1997newton}
Phong D H~ and Stein E M~.
\newblock The newton polyhedron and oscillatory integral operators.
\newblock {\em Acta Math}, 1997, 179: 105--152.

\bibitem{phong1998damped}
Phong D H~ and Stein E M~.
\newblock Damped oscillatory integral operators with analytic phases.
\newblock {\em Adv Math}, 1998, 134: 146--177.

\bibitem{ricci1987harmonic}
Ricci F and Stein E M~.
\newblock Harmonic analysis on nilpotent groups and singular integrals I.
  oscillatory integrals.
\newblock {\em J Funct Anal}, 1987, 73: 179--194.

\bibitem{rychkov2001sharp}
Rychkov V~S.
\newblock Sharp $L^2$ bounds for oscillatory integral operators with
  $C^\infty$ phases.
\newblock {\em Math Z}, 2001, 236: 461--489.

\bibitem{shi2016sharp}
Shi Z and Yan D.
\newblock Sharp $L^p$-boundedness of oscillatory integral operators with
  polynomial phases.
\newblock {\em Math Z}, 2017, 286: 1277--1302.

\bibitem{stein1993harmonic}
Stein  E~M and Murphy T~S.
\newblock {\em Harmonic analysis: real-variable methods, orthogonality, and
  oscillatory integrals}, volume~3.
\newblock Princeton University Press, 1993.

\bibitem{stein1959extension}
Stein E~M and Weiss G.
\newblock An extension of a theorem of marcinkiewicz and some of its
  applications.
\newblock {\em J Math Mech}, 1959, 8: 263--284.

\bibitem{tang2006decay}
Tang W.
\newblock Decay rates of oscillatory integral operators in ‘1+2’ dimensions.
\newblock {\em Forum Math}, 2006, 18: 427--444.

\bibitem{xiao2016endpoint}
Xiao L.
\newblock Endpoint estimates for one-dimensional oscillatory integral operator.
\newblock {Adv Math}, 2017, 316: 255--291.

\bibitem{yang2004sharp}
Yang C~W.
\newblock Sharp $L^p$ estimates for some oscillatory integral
  operators in $\mathbb{R}^1$.
\newblock {\em Illinois J. Math.}, 2004, 48: 1093--1103.

\end{thebibliography}
\end{document}